\documentclass{amsart}

 \usepackage{amsxtra}
 \usepackage{amsthm}
 \usepackage{amsmath}
 \usepackage{amsthm}
 \usepackage{amsfonts}
 \usepackage{amssymb}
 \usepackage[all]{xy}
\usepackage{hyperref} %to generate pdf with links

\input xy
\xyoption{all}
\newtheoremstyle{mytheorem}{.5em}{.5em}%
     {\it}%         Body font
     {}%         Indent amount (empty = no indent, \parindent = para indent)
     {}% Thm head font
     {}%        Punctuation after thm head
     {.5em}%     Space after thm head (\newline = linebreak)
     {#2 \thmname{ \bf{#1}.} \thmnote{\it{#3}.}}%         Thm head spec

\theoremstyle{mytheorem}
\newtheorem{theorem}{Theorem}[section]
\newtheorem{lemma}[theorem]{Lemma}
\newtheorem{corollary}[theorem]{Corollary}

\newtheoremstyle{note}{.5em}{.5em}%
     {}%         Body font
     {}%         Indent amount (empty = no indent, \parindent = para indent)
     {}% Thm head font
     {}%        Punctuation after thm head
     {.5em}%     Space after thm head (\newline = linebreak)
     {#2 \thmname{ \bf{#1}.} \thmnote{\it{#3}.}}%         Thm head spec

\theoremstyle{note}

\newtheorem{example}[theorem]{Example}
\newtheorem{remark}[theorem]{Remark}
\newtheorem{definition}[theorem]{Definition}

\newcommand{\Z}{\mathbb{Z}}       % integers
\newcommand{\qq}{\mathbb{Q}}      % rational numbers
\newcommand{\bbR}{\mathbb R}     % real numbers
\newcommand{\LL}{\mathbb{L}}       % Lazard ring
\newcommand{\CH}{\operatorname{CH}}

\newcommand{\calI}{\mathcal{I}}

\newcommand{\La}{\Lambda}           % Lambda
\newcommand{\la}{\lambda}            % lambda
\newcommand{\om}{\omega}          % omega

\newcommand{\eps}{\epsilon}            % epsilon
\newcommand{\veps}{\varepsilon}     % varepsilon
\newcommand{\Th}{\Theta}
\newcommand{\al}{\alpha}

\newcommand{\iso}{\stackrel{\simeq}}      % isomorphism
\newcommand{\too}{\longrightarrow}       % Long arrow

\newcommand{\IA}{\mathcal{I}_a}       % the augmentation ideal for F_a
\newcommand{\IM}{\mathcal{I}_m}     % the augmentation ideal for F_m
\newcommand{\IF}{\mathcal{I}_F}       % the augmentation ideal for F
\newcommand{\IFp}{\mathcal{I}_{F'}}   % the augmentation ideal for F'
\newcommand{\fgr}{R[[\La]]_F}            % formal group ring
\newcommand{\hgr}{\mathcal{H}(\La)_F}  % formal cohomology algebra

\newcommand{\hh}{\mathtt{h}}           % oriented cohomology
\newcommand{\ch}{\mathrm{c}^\hh}   % characteristic class
\newcommand{\cc}{\mathfrak{c}}        % characteristic map

\newcommand{\BB}{\mathfrak{B}}       % variety of Borel subgroups
\newcommand{\oct}{{o.c.t. }}           % oriented cohomology theory
\newcommand{\fgl}{{f.g.l. }}            % formal group law

\title{Invariants, exponents and formal group laws}

\date{\today}

\author{J.~Malag\'on-L\'opez, K.~Zainoulline, C.~Zhong}

\begin{document}

\maketitle

\begin{abstract}
Let $W$ be the Weyl group of a crystallographic root system acting on
the associated weight lattice by means of reflections. 
In the present notes we extend the notion of exponent of the
$W$-action to the context of an arbitrary algebraic oriented
cohomology theory of Levine-Morel and Panin-Smirnov and the associated
formal group law. 
From this point of view the classical Dynkin index of the associated Lie algebra will be the second exponent of the deformation map from the multiplicative to the additive formal group law. We apply this generalized exponent to study the torsion part of an arbitrary oriented cohomology theory of a twisted flag variety.
\end{abstract}

%%% Document STARTS here

%%%%%%%%%%%%%%%%%%%%%%%%%%%%%%%%%%%%%%%%%%%%%%%%%%%%%%%%%%%%%%%%%%
%%%%%%%%%%%%%%%%%%%%%%%%%%%%%%%%%%%%%%%%%%%%%%%%%%%%%%%%%%%%%%%%%%
\section{Introduction}
%%%%%%%%%%%%%%%%%%%%%%%%%%%%%%%%%%%%%%%%%%%%%%%%%%%%%%%%%%%%%%%%%%
%%%%%%%%%%%%%%%%%%%%%%%%%%%%%%%%%%%%%%%%%%%%%%%%%%%%%%%%%%%%%%%%%%

Let $W$ be the Weyl group of a crystallographic root system which acts
by means of simple reflections on the respective weight lattice $\La$.
Consider the induced actions of $W$ on the polynomial ring $S^*(\La)$
and the integral group ring $\Z[\La]$ of $\La$.
The ring of invariants $\Z[\La]^W$ can be identified with the
representation ring of the associated linear algebraic group $G$ and
according to the celebrated Chevalley theorem is a polynomial ring in
classes of {\em fundamental representations}.
On the other hand  $S^*(\La)^W\otimes\qq$ is known to be a
polynomial ring as well with generators given
by {\em basic polynomial invariants}.

The main goal of paper
\cite{ba&ne&za-ch} was to establish
the relationship between these two sets of invariants by means of the
{\em Chern class map}. This was done by introducing the notion of an
{\em exponent} -- an integer $\tau_d$ which measures the difference
between these invariants.
In particular, it was proven that $\tau_2$ coincides with the {\em Dynkin
  index} of the associated Lie algebra
(see Theorem~4.4 loc.cit.). Using the {\em Riemann-Roch}
theorem and the Grothendieck {\em $\gamma$-filtration} it was also shown
\cite[Corollary~6.8]{ba&ne&za-ch} that the exponent $\tau_d$
 bounds the annihilator of the torsion part of the Chow groups
$\mathrm{CH}^d$ for $d=2,3,4$ of some twisted flag
varieties, hence, providing new estimates for the torsion in small
codimensions.

The next step was done in \cite{ba&za&zh} where it was shown that
$\tau_d$ divides the Dynkin index $\tau_2$ for all $d\ge 2$. This fact was together with Demazure's description of the
kernel of the
characteristic map allowed to obtain a uniform bound for the
annihilator of the torsion of $\mathrm{CH}^d$ of strongly inner forms
of flag varieties for all $d$, hence, pushing geometric applications of the
exponent even further.

The goal of the present paper is to extend the notion of an exponent
to the context of arbitrary algebraic {\em oriented cohomology
  theories} (o.c.t.) and the associated {\em formal group laws}
(f.g.l.). From this
point of view, the Dynkin index of a root system will be
just the second exponent of the deformation from
the multiplicative to the additive \fgl We also expect (see the last section)
that these generalized exponents can be used to estimate the torsion part of an
arbitrary algebraic \oct

Recall that the notion of an algebraic o.c.t. was
introduced
by  Levine-Morel \cite{lev&mor-book} and by Panin-Smirnov \cite{pa&sm}.
Roughly speaking, it is a cohomological-type functor $\hh$ from the
category of smooth varieties over a field to the the category of
commutative rings endowed with push-forward maps and characteristic
classes (see \S\ref{sec:orcoh}).
Basic examples of such functors are the Chow ring of algebraic cycles
modulo the rational
equivalence relation, the Grothendieck $K_0$, the algebraic cobordism
$\Omega$ of Levine-Morel (see \cite[\S2.1, 2.5, 3.8]{pa&sm} for more
examples).
The theory of formal group laws  originated from the theory of Lie
groups. Roughly speaking, a \fgl  is a formal power series
which behaves as if it were the product of a Lie group.
Its subclass of one-dimensional commutative formal group laws  had been intensively
used in topology, especially in cobordism theory and, more generally,
for studying topological oriented theories.
A link between oriented cohomology and formal group laws  is given
by the formula expressing the first characteristic class $\ch_1$ of a
tensor product of two line bundles $\ch_1(L_1\otimes
L_2)=F(\ch_1(L_1),\ch_1(L_2))$, where $F$ is the one-dimensional
commutative \fgl  over the coefficient ring $R$ associated
to $\hh$.

The key notion of the paper can be described as follows:
According to \cite{ca&za&pe-tor} the algebras $S^*(\La)$ and $\Z[\La]$
can be viewed as
specializations (for additive and multiplicative $F$ resp.) of
a more general object -- the formal
group algebra  $R[[\La]]_F$.
Consider its subring of invariants $R[[\La]]_F^W$ and the associated
ideal $\IF^W$.
Given two formal group laws  $F$
and $F'$ there is a (non canonical and non $W$-equivariant) isomorphism
$\Phi\colon R[[\La]]_F \iso\too R[[\La]]_{F'}$ called the
{\em deformation map}.
The  {\em exponent} $\tau_d^{F\to F'}$ from $F$ to $F'$ is defined to
be the exponent of the image $\Phi(\IF^W)$ in $\IFp^W$ after passing
to the $d$-th subsequent quotients of the $\IF$- and $\IFp$-adic
filtrations (see Def.~\ref{defFexp}). In case $F$ is the
multiplicative and $F'$ is the additive \fgl  we obtain the
exponent defined in \cite{ba&ne&za-ch}.

The paper is organized as follows:
In Section~\ref{sec:inv} we recall notation and results of
\cite{ba&ne&za-ch}. In Section~\ref{sec:form-group} we recall
the definition of a formal group algebra(ring) following
\cite[\S2]{ca&za&pe-tor}.
In Section~\ref{sec:deformation-maps} we define and study deformation
maps between formal group algebras (see~\ref{rem-phi-add}).
In Section~\ref{sec:fexp} we introduce the notion of an exponent
between formal group laws  (see~\ref{defFexp}) and we
prove its existence (Theorem~\ref{thmexs}); we also explain how to
compute generators of the ring of invariants rationally.
In Section~\ref{sec:orcoh} we recall several facts concerning
oriented cohomology and characteristic classes; we provide examples of computations
involving characteristic classes (see~\ref{lem-comp-cc} and
\ref{prop-pro}).
In Section ~\ref{geomthm} we establish a link between
characteristic classes and the deformation maps
(Theorem~\ref{geomthm}), hence, providing a geometric interpretation
for the latter.
In the last section we compute exponents $\tau_d^{F\to F'}$ for some
root systems (Theorem~\ref{ABD}) and provide application to the
problem of estimating the
torsion part of an \oct (Corollaries~\ref{torsapp} and \ref{torsch})

%%%%%%%%%%%%%%%%%%%%%%%%%%%%%%%%%%%%%%%%%%%%%%%%%%%%%%%%%%%%%%%%%%
%%%%%%%%%%%%%%%%%%%%%%%%%%%%%%%%%%%%%%%%%%%%%%%%%%%%%%%%%%%%%%%%%%
\section{Invariants and exponents}\label{sec:inv}
%%%%%%%%%%%%%%%%%%%%%%%%%%%%%%%%%%%%%%%%%%%%%%%%%%%%%%%%%%%%%%%%%%
%%%%%%%%%%%%%%%%%%%%%%%%%%%%%%%%%%%%%%%%%%%%%%%%%%%%%%%%%%%%%%%%%%

In the present section we recall basic definitions and results of
\cite{ba&ne&za-ch}.

%%%%%%%%%%%%%%%%%%%%%%%%%%%%%%%%%%%%%%%%%%%%%%%%%
\

Let $\La$ be a free abelian group of finite rank.
Consider the integral group ring $\Z[\La]$ of $\La$.
Its elements are finite linear combinations
$\sum_j a_j e^{\la_j}$, $a_j\in \Z$, $\la_j\in \La$.
Consider the augmentation map
$\eps_m\colon \Z[\La]\to \Z$ given by $e^\la \mapsto 1$.
Its kernel $I_m$ is an ideal in $\Z[\La]$ generated by the
differences $(1-e^{-\la})$, $\la\in \La$.
Consider the $I_m$-adic filtration on $\Z[\La]$
\[
  \Z[\La] = I_m^0\supseteq I_m\supseteq I_m^2 \supseteq \ldots
\]
Similarly, consider the symmetric algebra $S^*(\La)$ of $\La$ and
the augmentation map $\eps_a\colon S^*(\La)\to \Z$
which sends any polynomial to its constant term.
Its kernel $I_a$ is an ideal generated by $\La=S^1(\La)$.
Consider the $I_a$-adic filtration on $S^*(\La)$
$$
S^*(\La)=I_a^0\supseteq I_a \supseteq I_a^2 \supseteq \ldots.
$$

%%%%%%%%%%%%%%%%%%%%%%%%%%%%%%%%%%%%%%%%%%%%%%%%%%%%

The rings $\Z[\La]$ and $S^*(\La)$ are non-isomorphic.
However, they become isomorphic after truncation.
Indeed, let $\{\om_1,\ldots,\om_n\}$ be a $\Z$-basis of $\La$,
then for each $d \ge 0$ we have two reverse  isomorphisms defined by
(cf. \cite[Lem.~1.2]{ba&ne&za-ch} and
\cite[Def.~2.1]{ga&ki-gamma})
\[
\psi\colon
S^*(\La)/I_a^{d+1} \iso\rightleftarrows \Z[\La]/I_m^{d+1}
\colon \phi
\]
\[
  \psi(\om_j)=(1-e^{-\om_j})
\;\text{ and }\;
  \phi(e^{\sum_{j=1}^n a_j\om_j})
=
  \prod_{j=1}^n(1-\om_j)^{-a_j}.
\]
The morphisms $\phi$ and $\psi$ preserve
the $I_a$- and $I_m$-adic filtrations. Set $I_m^{(d)}=I_m^d/I_m^{d+1}$
and $I_a^{(d)}=I_a^d/I_a^{d+1}=S^d(\La)$.
Restricting $\phi$ and $\psi$ to the subsequent quotients
we obtain isomorphisms \cite[Lem.~1.3]{ba&ne&za-ch}
\[
  \psi_{d}\colon
I_a^{(d)}\, \iso\rightleftarrows\, I_m^{(d)}
\colon \phi_{d},\text{ where }
  \psi_{d}\left( \prod_{j=1}^d \la_j \right)
=
  \prod_{j=1}^d (1-e^{-\la_j}),\; \la_j\in \La.
\]
Observe that the isomorphisms $\phi$, $\psi$ depend on the choice
of a basis of $\La$ and the isomorphisms $\phi_d$, $\psi_d$ do not.

%%%%%%%%%%%%%%%%%%%%%%%%%%%%%%%%%%%%%%%%%%%%%%%%%%%%%%

\medskip

Let $W$ be a finite group which acts on $\La$
by $\Z$-linear automorphisms.
Consider the subrings of invariants $S^*(\La)^W$ and $\Z[\La]^W$.
Let $I_a^W$ and $I_m^W$ denote the ideals generated by elements of
$S^*(\La)^W\cap I_a$ and $\Z[\La]^W\cap I_m$ respectively.

The $W$-action is compatible with the $I_a$- and
$I_m$- adic filtrations, i.e. $W(I_a^d)\subseteq I_a^d$ and
$W(I_m^d)\subseteq I_m^d$ for all $d \ge 0$.
Observe that the isomorphisms $\phi_{d}$ and $\psi_{d}$
are $W$-equivariant
(in particular, we have $(I_a^{(d)})^W \simeq (I_m^{(d)})^W$)
but the isomorphisms $\phi$ and $\psi$ are not.

%%%%%%%%%%%%%%%%%%%%%%%%%%%%%%%%%%%%%%%%%%%%%%%%%%%

Let $(I_a^W)^{(d)}\subseteq S^d(\La)$ denote the $d$-th homogeneous
component of the ideal
$I_a^W$ and let $(I_m^W)^{(d)}=(I_m^W\cap I_m^d)/(I_m^W\cap
I_m^{d+1})$
denote the $d$-th homogeneous component of the ideal $I_m^W$.
Since the isomorphisms
$\phi$ and $\phi_{d}$ preserve the filtrations, we have
\[
\phi(I_m^W/I_m^{d+1})\cap I_a^{(d)}=\phi_{d} \left((I_m^W)^{(d)}\right).
\]

\begin{definition}\label{expon}
We say that an action of $W$ on $\La$ has a finite exponent in degree
$d$ (see \cite[Def.~2.1]{ba&ne&za-ch})
if there exists a non-zero integer $N_d$ such that
$$
N_d \cdot (I_a^W)^{(d)} \subseteq
\phi_{d} \left((I_m^W)^{(d)}\right).
$$
In this case the $g.c.d.$ of all such $N_d$-s is called
the {\em $d$-th exponent} of the $W$-action and
is denoted by $\tau_d$.
\end{definition}

Observe that if $\phi_{d} \left((I_m^W)^{(d)}\right)$
is a subgroup of finite index in $(I_a^W)^{(d)}$,
then  $\tau_d$ is simply its exponent.

%%%%%%%%%%%%%%%%%%%%%%%%%%%%%%%%%%%%%%%%%%%%%%%%%%%%
\begin{remark}
The main result of \cite{ba&ne&za-ch}
says that in the case of a crystallographic root system
an action of the Weyl group $W$ on the weight lattice $\Lambda$
has finite exponents in each degree (Cor.~6.3 loc.cit.) and
$\tau_2$ coincides with the Dynkin index
of the associated Lie algebra (Thm.~4.4 loc.cit.).
Moreover, according to \cite[Prop.~5.6]{ba&za&zh}, if the root system is of type $B_n$ or 
$D_n$, then 
we have $\tau_d\mid \tau_2$ for all $d\ge 3$.
\end{remark}

%%%%%%%%%%%%%%%%%%%%%%%%%%%%%%%%%%%%%%%%%%%%%%%%%%%%%%%%%%%%%%%%%%
%%%%%%%%%%%%%%%%%%%%%%%%%%%%%%%%%%%%%%%%%%%%%%%%%%%%%%%%%%%%%%%%%%
\section{Formal group algebras}\label{sec:form-group}
%%%%%%%%%%%%%%%%%%%%%%%%%%%%%%%%%%%%%%%%%%%%%%%%%%%%%%%%%%%%%%%%%%
%%%%%%%%%%%%%%%%%%%%%%%%%%%%%%%%%%%%%%%%%%%%%%%%%%%%%%%%%%%%%%%%%%

In the present section we recall definition and basic properties of
\fgl  \cite[Ch.~IV]{Ell} and
formal group algebras \cite[\S2]{ca&za&pe-tor}.

%%%%%%%%%%%%%%%%%%%%%%%%%%%%%%%%%%%%%%%%%%%%%%%%%%%%%%

\medskip

By $F$ we always denote
a commutative one-dimensional \fgl
over a commutative ring $R$ called the coefficient ring of $F$,
i.e. $F$ is a power series in two variables
\[
  F(u,v)
=
  u + v + \sum_{i,j\ge 1} a_{ij}u^iv^j,
\quad
  a_{ij} = a_{ji},
\quad
  a_{ij}\in R
\]
which satisfies axioms of a group law.

A morphism of \fgl
$f\colon F \to F'$ over $R$ is a power-series $f(u)=u+O(2)\in R[[u]]$
such that $f(F(u,v))=F'(f(u),f(v))$.
For any \fgl  $F$
there is a formal inverse ($F$-inverse) $\imath_F (u) \in R[[u]]$
which is defined by the identity
$F( u , \imath_F(u))  = 0$.

We will use the following notation
\[
u+_F v=F(u,v),\; -_F u=\imath_F(u)
\text{ and }
a\cdot_{F} u =
\underbrace{u +_F u +_F \cdots +_F u}_{a\; times},\; a\ge 1.
\]
\begin{example}
(a)
The additive \fgl is given by $F_a (u,v) = u + v$.

\noindent
(b)
The multiplicative \fgl is given by $F_m (u,v) = u + v - uv$.

\noindent
(c)
The Lorentz \fgl is given by
\[
  F_l(u,v)
=
  \tfrac{u+v}{1+uv}
=
  u + v + \sum_{i \geq 1} (-1)^i \left( u^i v^{i+1} + u^{i+1} v^i \right).
\]

\noindent
(d)
Let $E$ be an elliptic curve defined by
(see  \cite[IV.1]{Ell})
\[
  E:
\qquad
  v
=
  u^3 + a_1 uv + a_2 u^2 v + a_3 v^2 + a_4 u v^2 + a_6 v^3,\; a_i\in\Z.
\]
The group law on $E$ induces an elliptic \fgl
\begin{eqnarray*}
  F_e (u, v)
&=&
  u + v - a_1 u v - a_2 (u^2 v + v^2 u)+
\\
&\,&
 2 a_3 (u^3 v + u v^3) + (a_1 a_2 - 3 a_3) u^2 v^2 + O(5).
\end{eqnarray*}

\noindent
(e)
There is a universal \fgl $F_u$.
Its coefficient ring is called the Lazard ring $\LL$.
Any commutative one-dimensional \fgl
over a ring $R$ corresponds to a ring homomorphism from $\LL$ to $R$.
\end{example}

%%%%%%%%%%%%%%%%%%%%%%%%%%%%%%%%%%%%%%%%%%%%%%%%%%%%
\begin{definition}(cf. \cite[Def.~2.4]{ca&za&pe-tor})\label{defFgral}
Let $F$ be a \fgl  over a coefficient ring $R$,
and let $\La$ be an Abelian group.
Consider the polynomial ring $R[x_\La]$ in variables $x_\la$, $\la \in\La$,
and let $\eps \colon R[x_\La] \to R$, $x_\la \mapsto 0$, be the
augmentation map.
Let $R[[x_\La]]$ be the $\ker\eps$-adic completion of $R[x_\La]$.
Let $\mathcal{J}_{F}$ be the closure of the ideal of $R[[x_\La]]$
generated by $x_0$ and elements of the form
$x_{\la_1 + \la_2} - \left( x_{\la_1} +_F x_{\la_2} \right)$
for all $\la_1, \la_2\in \La$.
The \emph{formal group algebra (ring)} $\fgr$ is defined to be
the quotient
%(see \cite[Def.~2.4]{ca&za&pe-tor})
\[
  \fgr := R [[x_\La]]/\mathcal{J}_{F}.
\]
The class of $x_\la$ in $\fgr$ will be denoted by the same letter.
Notice that we have $x_{-\la}=\imath(x_\la)=-_F x_\la$ in $\fgr$.
\end{definition}

%%%%%%%%%%%%%%%%%%%%%%%%%%%%%%%%%%%%%%%%%%%%%%%%%%%%%%
\begin{example}
Let $F_a(u,v)=u+v$ be the additive \fgl  over $\Z$.
The induced formal group ring over $\La$ will be denoted as
$\Z[[\La]]_a$.
There is a ring isomorphism
\begin{equation}\label{ex-FGR-add}
  \Z[[\La]]_a \iso\too \prod_{d \geq 0} S^d (\La), \quad x_\la \mapsto \la.
\end{equation}
Let $F_m(u,v)=u+v-uv$ be the multiplicative \fgl  over $\Z$.
The induced formal group ring will be denoted as $\Z[[\La]]_m$.
There is a ring isomorphism
\[
  \Z[[\La]]_{m} \iso\too \Z[\La]^\wedge,\quad x_\la \mapsto (1 - e^{-\la}),
\]
where $\Z[\La]^\wedge$ denotes the completion
of the usual group ring $\Z[\La]$ along the augmentation ideal.
\end{example}

%%%%%%%%%%%%%%%%%%%%%%%%%%%%%%%%%%%%%%%%%%%%%%%%%%%%%%%
%%%%%%%%%%%%%%%%%%%%%%%%%%%%%%%%%%%%%%%%%%%%%%%%%%%%%%%
\section{Deformation maps}\label{sec:deformation-maps}
%%%%%%%%%%%%%%%%%%%%%%%%%%%%%%%%%%%%%%%%%%%%%%%%%%%%%%%
%%%%%%%%%%%%%%%%%%%%%%%%%%%%%%%%%%%%%%%%%%%%%%%%%%%%%%%

In the present section we introduce the completed analogues (called
the deformation maps) of the maps
$\phi$, $\psi$, $\phi_{d}$, $\psi_{d}$ of Section~\ref{sec:inv}.

\begin{definition}\label{rem-phi-add}
Given two \fgl  $F$ and $F'$ over $R$
we define an $R$-algebra homomorphism
$\Phi^{F\to F'}\colon \fgr \to R[[\La]]_{F'}$ as follows

First, we set
$\Phi^{F\to F'}(x_0):=0$ and $\Phi^{F\to F'}(x_{\om_i}):=x_{\om_i}$
for each basis element $\om_i$ of $\La$.
Then for every $\la = \sum_{i=1}^n a_i \om_i$ with $a_i \in \Z$
we define
\[
  \Phi^{F\to F'}( x_\la)
: =
  (a_1 \cdot_{F} x_{\om_1}) +_{F} (a_2 \cdot_{F} x_{\om_2}) +_{F}
\ldots +_{F} (a_n \cdot_{F} x_{\om_n}).
\]
Finally, we set
\[
  \Phi^{F\to F'}(x_\la + x_{\la'})
:=
  \Phi^{F\to F'}(x_\la) + \Phi^{F\to F'}(x_{\la'})\text{ and}
\]
\[
  \Phi^{F\to F'}(x_\la\cdot x_{\la'})
:=
  \Phi^{F\to F'}(x_\la)\cdot \Phi^{F\to F'}(x_{\la'}).
\]

The map $\Phi^{F\to F'}$ will be called a {\em deformation map}
from the formal group algebra $\fgr$ to $R[[\La]]_{F'}$.

\medskip

By definition the composite $\Phi^{F'\to F}\circ \Phi^{F\to F'}$ is
the identity map.
Therefore, $\Phi^{F\to F'}$ is an isomorphism with $\Phi^{F'\to F}$
being the inverse.
Observe that all formal group algebras are (non-canonically)
isomorphic to $R[[\om_1,\ldots,\om_n]]$ via the deformation map
$\Phi^{F\to F_a}$ (cf. \cite[Cor.~2.12]{ca&za&pe-tor}).
\end{definition}

%%%%%%%%%%%%%%%%%%%%%%%%%%%%%%%%%%%%%%%%%%%%%%%%%%%%
\begin{example}
Consider the deformation maps $\Phi^{F_a\to F}$ and $\Phi^{F_m\to F}$
from the formal group algebras corresponding
to the additive and multiplicative \fgl respectively.
Let $\la=\sum_{i=1}^n a_i\om_i$.

Since $a\cdot_{F_m} u=1-(1-u)^a$ for $a\in \Z$, we obtain
\begin{eqnarray*}
  \Phi^{F_m\to F} (x_\la)
&=&
  \left( 1 - (1 - x_{\om_1})^{a_1} \right)
+_{F_m}
  \cdots
+_{F_m}
  \left( 1 - (1 - x_{\om_n})^{a_n} \right)
\\
&=&
  1 - (1 - x_{\om_1})^{a_1} \cdots (1 - x_{\om_n})^{a_n}.
\end{eqnarray*}
Since $a\cdot_{F_a} u=a\cdot u$, we obtain
\[
\Phi^{F_a\to F}(x_\la)=a_1x_{\om_1}+a_2x_{\om_2}+\ldots + a_nx_{\om_n}.
\]
\end{example}

%%%%%%%%%%%%%%%%%%%%%%%%%%%%%%%%%%%%%%%%%%%%%%%%%%%%%%%

Let $\IF$ denote the kernel of the augmentation map
$\eps_F\colon \fgr \to R$, given by $x_\la \mapsto 0$.
Consider the $\IF$-adic filtration
\[
  \fgr =\IF^0\supseteq\IF\supseteq\IF^2\supseteq \cdots
\]
We denote by $\IF^{(d)}=\IF^d/\IF^{d+1}$ the $d$th subsequent quotient.
The deformation map $\Phi^{F\to F'}$
respects the $\IF$- and $\IFp$-adic filtrations, i.e.
\[
\Phi^{F\to F'}(\IF^d) = \IFp^d,\quad d\ge 1,
\]
hence, it induces ring isomorphisms on quotients
\begin{equation}\label{eq-iso-slices}
\Phi^{F\to F'}\colon \fgr/\IF^{d+1} \iso\too R[[\La]]_{F'}/\IFp^{d+1}\;
\text{ and }\;
\Phi^{F\to F'}_{d}\colon \IF^{(d)} \iso\too \IFp^{(d)}.
\end{equation}
Observe that contrary to $\Phi^{F\to F'}$ the isomorphism
$\Phi^{F\to F'}_{d}$
does not depend on the choice of a basis of $\La$.
Indeed, $\IF^{(d)}$ is $R$-linearly generated by
$\prod_{i=1}^dx_{\la_i}$, $\la_i\in \La$, and
\[
\Phi^{F\to F'}_{d}\left(\prod_{i=1}^d
x_{\la_i}\right)=\prod_{i=1}^d x_{\la_i}.
\]

%%%%%%%%%%%%%%%%%%%%%%%%%%%%%%%%%%%%%%%%%%%%%%%%%%%
\begin{example}
For $F=F_a$ the isomorphism
$\Phi^{F\to F'}_{d}$ coincides with the isomorphism of
\cite[Lemma~4.2]{ca&za&pe-tor}.
\end{example}

%%%%%%%%%%%%%%%%%%%%%%%%%%%%%%%%%%%%%%%%%%%%%%%%%%%%%%%%%%%%%%%%%%

Note that the symmetric algebra $S^*(\La)$
can be identified with the image of
\[
\veps_a\colon S^* (\La) \too \Z[[\La]]_a,
\qquad
\om_i \mapsto x_{\om_i},
\]
and the integral group ring $\Z[\La]$
can be identified with the image of
\[
\veps_m\colon \Z[\La] \too \Z[[\La]]_m,
\qquad
  e^\la \mapsto (1 - x_{-\la}).
\]
The maps $\veps_a$ and $\veps_m$ preserve the filtrations, i.e.
$\veps_a(I_a)\subseteq \IA$ and $\veps_m(I_m)\subseteq \IM$,
where $\IA=\mathcal{I}_{F_a}$ and $\IM=\mathcal{I}_{F_m}$.
Therefore, they induce isomorphisms on quotients
\begin{equation}\label{truniso}
\veps_a\colon S^*(\La)/I_a^{d+1} \iso\too \Z[[\La]]_a/\IA^{d+1}\text{ and }
\veps_m\colon \Z[\La]/I_m^{d+1} \iso\too \Z[[\La]]_m/\IM^{d+1}.
\end{equation}
There are commutative diagrams
\[
 \xymatrix{
            \Z [\La]/I_m^{d+1} \ar@<.5ex>[d]^{\phi} \ar[r]^{\veps_m}&
            \Z [[\La]]_m / \IM^{d+1}
              \ar@<.5ex>[d]^{\Phi^{F_m\to F_a}} \\
            S^*(\La)/I_a^{d+1} \ar@<.5ex>[u]^{\psi} \ar[r]^{e_a} &
            \Z [[\La]]_a / \IA^{d+1}
              \ar@<.5ex>[u]^{\Phi^{F_a\to F_m}},
 }
\qquad
 \xymatrix{
            I_m^{(d)} \ar@<.5ex>[d]^{\phi_{d}} \ar[rr]^-{\veps_m}& &
            \IM^{(d)}
              \ar@<.5ex>[d]^{\Phi^{F_m\to F_a}_{d}} \\
            I_a^{(d)} \ar@<.5ex>[u]^{\psi_{d}} \ar[rr]^-{\veps_a} & &
           \IA^{(d)}
              \ar@<.5ex>[u]^{\Phi^{F_a\to F_m}_{d}}.
 }
\]
In other words, the deformation maps between $F_m$ and $F_a$
can be viewed as completions of the maps $\phi$, $\phi_d$ and $\psi$,
$\psi_d$.

%%%%%%%%%%%%%%%%%%%%%%%%%%%%%%%%%%%%%%%%%%%%%%%%%%%%%%%%%%%%%%%%%%
%%%%%%%%%%%%%%%%%%%%%%%%%%%%%%%%%%%%%%%%%%%%%%%%%%%%%%%%%%%%%%%%%%
\section{Exponents between formal group laws}\label{sec:fexp}
%%%%%%%%%%%%%%%%%%%%%%%%%%%%%%%%%%%%%%%%%%%%%%%%%%%%%%%%%%%%%%%%%%
%%%%%%%%%%%%%%%%%%%%%%%%%%%%%%%%%%%%%%%%%%%%%%%%%%%%%%%%%%%%%%%%%%

In the present section we introduce the notion of an exponent between \fgl
(see~\ref{defFexp}) and prove its existence (see~\ref{thmexs}).

%%%%%%%%%%%%%%%%%%%%%%%%%%%%%%%%%%%%%%%%%%%%%%%%%%%%%%%

\medskip

As in Section~\ref{sec:inv} let $W$ be a finite group which acts on
$\La$ by $\Z$-linear automorphisms.
It induces an action on the associated formal group algebra $\fgr$
(see~\ref{defFgral}).
Let $\fgr^W$ denote the subring of invariants and
let $\IF^W$ denote the ideal generated by $\fgr^W \cap \IF$.

As an immediate consequence of the definition the set $\fgr^W \cap
\IF$ consists of formal sums of elements of the form
\[
  \rho \left( x_{\eta_1} \cdots x_{\eta_s} \right)
=
  \sum_{x_{\la_1} \cdots x_{\la_s} \in W(x_{\eta_1} \cdots  x_{\eta_s})}
   x_{\la_1}\cdot \ldots\cdot x_{\la_s},
\quad
  s \geq 1.
\]

Let $(\IF^W)^{(d)}=(\IF^W\cap \IF^d)/(\IF^W\cap \IF^{d+1})$
denote the $d$-th homogeneous component of the ideal $\IF^W$ with
respect to the $\IF$-adic filtration.

%%%%%%%%%%%%%%%%%%%%%%%%%%%%%%%%%%%%%%%%%%%%%%%%%%%
\begin{definition}\label{defFexp}
Let $F$ and $F'$ be \fgl  over a ring $R$.
We say that an action of $W$ on $\La$
has a {\em finite exponent from $F$ to $F'$} in degree $d$
if there exists a positive integer $N^{F\to F'}_d$ coprime to the characteristic of $R$ and  such that
\begin{equation}\label{Fexp}
N_d^{F\to F'} \cdot (\IFp^W)^{(d)}\; \subseteq\;
\Phi^{F\to F'}_{d}\left((\IF^W)^{(d)}\right).
\end{equation}
In this case the smallest such $N_d^{F\to F'}$-s will be called
the {\em $d$-th exponent from $F$ to $F'$ of the $W$-action}
and will be denoted by $\tau_d^{F\to F'}$.
\end{definition}

Observe that
$\Phi^{F\to F'}(\IF^W/\IF^{d+1})\cap \IFp^{(d)}=\Phi^{F\to
  F'}_{d}\left((\IF^W)^{(d)}\right)$.

According to \cite[Lemma 4.2]{ca&za&pe-tor} for every $d$
there is an isomorphism of $R$-modules
\[
  S^d_R(\La)
=
  S^d(\La)\otimes_\Z R \iso\too \IF^{(d)},\qquad \la\mapsto x_\la.
\]
Therefore, the subgroups of \eqref{Fexp} can be identified
with subgroups in $S_R^d(\La)$.

Since $\veps_a$ and $\veps_m$ induce isomorphisms on the truncations
\eqref{truniso}, they induce isomorphisms on the $d$-th homogeneous
components
$(I_a^W)^{(d)} \iso\too (\IA^W)^{(d)}$ and $(I_m^W)^{(d)} \iso\too
(\IM^W)^{(d)}$.
The latter implies that
a $W$-action has a finite exponent from $F_m$ to $F_a$ if and only if
it has a finite exponent in the sense of Definition~\ref{expon}.
Moreover, the exponent $\tau_d^{F_m\to F_a}$ and
the exponent $\tau_d$ from \ref{expon} coincide.

%%%%%%%%%%%%%%%%%%%%%%%%%%%%%%%%%%%%%%%%%%%%%%%%%%%%%%%%%%%%%%%%%%
\begin{example}\label{example:tau1}
We have $\tau_1^{F\to F'} = 1$ for any \fgl  $F$ and $F'$.
Indeed, regarding a monomial $x_{\la_1} \cdots x_{\la_s}$ as a
power series in the $x_{\om_i}$'s we have that
$\deg (x_{\la_1} \cdots x_{\la_s}) \geq s$.
Thus, we only need to consider the $W$-invariants $\rho
(x_\la)$, $\la \in \La$.
Since $F = F_a+O(2)$ for any \fgl  $F$, the homogenous components of
degree one of $\rho (x_\la)$ in $\mathbb{Z} [[\La]]_F$ and
$\mathbb{Z} [[\la]]_a$ are the same.
Since $\IA^W/\IA^2= 0$ (see \cite[\S3]{ba&ne&za-ch}),
the statement follows.
\end{example}

%%%%%%%%%%%%%%%%%%%%%%%%%%%%%%%%%%%%%%%%%%%%%%%%%

Given a \fgl  $F$ over  a ring $R$ of characteristic 0
there is always an isomorphism of \fgl
after tensoring with $\qq$
\[
  \log_F
\colon
  F_\qq
\iso\too
  (F_a)_{R_\qq}
\]
given by the logarithm series $\log_F(u)\in R_\qq[[u]]$.
The latter is defined by
\[
  \log_F(u)
:=
  \int_0^u \frac{dv}{g(v)},
\text{ where }
  g(u)
=
  \frac{\partial F(u,v)}{\partial v}|_{v=0}.
\]
and satisfies $\log_F(u+_F v)=\log_F(u) + \log_F(v)$
(see \cite[Ch.~IV.5]{Ell}).

%%%%%%%%%%%%%%%%%%%%%%%%%%%%%%%%%%%%%%%%%%%%%%
\begin{example}
If $F=F_m$, i.e. $F_m(u,v)=u+v-uv$, then
\[
  \log_{F_m}(u)
=
  \log(1-u)=u+\tfrac{u^2}{2}+\tfrac{u^3}{3}+\ldots
\]
For a general \fgl  $F(u,v)=u+v+a_{11}uv+O(3)$ we have
\[
  \log_F(u)
=
  u-\tfrac{a_{11}}{2} u^2+O(3).
\]
\end{example}

%%%%%%%%%%%%%%%%%%%%%%%%%%%%%%%%%%%%%%%%%%%%%%%%%

A morphism of \fgl  $f\colon F\to F'$ over $R$ induces a
homomorphism of formal group rings
$f^\star\colon R[[\La]]_{F'}\to \fgr$, $x_\la \mapsto f(x_\la)$
(see \cite[\S2.5]{ca&za&pe-tor}).
By definition $f^\star$ is $W$-equivariant and preserves the $\IF$-adic
filtration.
Therefore, we have induced maps
\[
  f^\star\colon R[[\La]]^W_{F'} \to \fgr^W
\text{ and }
  f^\star\colon (\IFp^W)^{(d)} \to (\IF^W)^{(d)}.
\]
In particular, after tensoring with $\qq$ we obtain the following ring
isomorphisms
\[
  \log_F^\star
\colon
  R_\qq[[\La]]_a^W
\iso\too
  R_\qq[[\La]]_F^W
\text{ and }
  \log_F^{\star}
\colon
  (\IA^W)^{(d)}_{R_\qq}
\iso\too
  (\IF^W)^{(d)}_\qq,
\]
where the last map is defined by
$\la_1\cdots \la_d  \mapsto x_{\la_1}\cdots x_{\la_d}$.
Observe that by definition $\log_F^{\star}$ coincides with the
deformation map $\Phi^{F_a\to F}_{d}$ restricting on
$(\IA^W)^{(d)}_{\qq}$.

The following theorem generalizes \cite[Cor.~6.3]{ba&ne&za-ch}

\begin{theorem}\label{thmexs}
Let $F$, $F'$ be \fgl  defined over a  ring of characteristic 0.
Then a $W$-action on $\La$ has a finite exponent from $F$ to $F'$.
\end{theorem}

\begin{proof}
The composite
$\log_{F'}^\star\circ (\log_F^\star)^{-1}\colon
(\IF^{(d)})_\qq\to (\IFp^{(d)})_\qq$
coincides with $\Phi^{F\to F'}_{d}$
and its restriction gives an isomorphism
$(\IF^W)^{(d)}_\qq\iso\too (\IA^W)^{(d)}_\qq\iso\too
(\IFp^W)^{(d)}_\qq$.
Finally, observe that $(\IA^W)^{(d)}_\qq$ is finitely generated,
so the exponents $\tau_d^{F\to F'}$ exist.
\end{proof}

%%%%%%%%%%%%%%%%%%%%%%%%%%%%%%%%%%%%%%%%%%%%%%%%%%%%%%%%%%%%%%%%%%
%%%%%%%%%%%%%%%%%%%%%%%%%%%%%%%%%%%%%%%%%%%%%%%%%%%%%%%%%%%%%%%%%%
\section{Oriented cohomology and characteristic classes}\label{sec:orcoh}
%%%%%%%%%%%%%%%%%%%%%%%%%%%%%%%%%%%%%%%%%%%%%%%%%%%%%%%%%%%%%%%%%%
%%%%%%%%%%%%%%%%%%%%%%%%%%%%%%%%%%%%%%%%%%%%%%%%%%%%%%%%%%%%%%%%%%

In the present section we recall several auxiliary facts
concerning \oct
We refer to \cite{lev&mor-book} and \cite{pa&sm} for details and
examples.

%%%%%%%%%%%%%%%%%%%%%%%%%%%%%%%%%%%%%%%%%%%%%%%%%%%%%%%%

\medskip

An \oct is a contravariant functor $\hh$
from the category of smooth projective varieties over a field $k$
to the category of commutative unital rings
which satisfies certain properties \cite[\S1.1]{lev&mor-book}.
Given a morphism $f\colon X\to Y$
the functorial map $\hh(f)$ will be denoted by $f^*$
and called the {\em pull-back}.
One of the characterizing properties of $\hh$ says that
for any proper map $f\colon X\to Y$
there is an induced map $f_*\colon \hh(X) \to \hh(Y)$ of
$\hh(Y)$-modules called the {\em push-forward}
(here $\hh(X)$ is an $\hh(Y)$-module via $f^*$).
A \emph{morphism} of \oct
is a natural transformation of functors
that also commutes with the push-forwards.

%%%%%%%%%%%%%%%%%%%%%%%%%%%%%%%%%%%%%%%%%%%%%%%%%%%%

\medskip

Another characterizing property of an \oct is the existence of
characteristic classes.
The latter is a collection of maps
$\ch_i\colon K_0(X)\to \hh(X)$, $i\ge 1$
that satisfy the following properties:

\medskip

\noindent
Let $\ch(x) = 1 + \ch_1(x)t + \ch_2(x)t^2+\ldots \in \hh(X)[[t]]$
denote the total characteristic class. Then
\begin{itemize}\label{charclasses}
\item[(a)] $\ch(E)=1$ for a trivial bundle $E$ over $X$,
\item[(b)] $\ch_i(E)=0$ for a bundle $E$ with $i> rk(E)$,
\item[(c)] $\ch(E\oplus F)=\ch(E)\cdot \ch(F)$
for any two bundles $E$ and $F$ over $X$.
\end{itemize}

Given two line bundles $L_1$ and $L_2$ over $X$,
we have \cite[Lem.~1.1.3]{lev&mor-book}
\begin{equation}\label{equ-fgl&oct}
 \ch_1 (L_1 \otimes L_2)
=
  \ch_1 (L_1) +_F \ch_1 (L_2),
\end{equation}
where $F$ is a one-dimensional commutative \fgl over the coefficient
ring $R=\hh(Spec(k))$ associated to $\hh$.

There is an \oct $\Omega$
defined over a field of characteristic zero
\cite[\S1.2]{lev&mor-book}, called \emph{algebraic cobordism},
that is universal in the following sense:
Given any \oct $\hh$ there is a unique morphism
$\Theta_\hh \colon \Omega \to \hh$ of \oct
Observe that the \fgl associated to $\Omega$
is the universal \fgl $F_u$.

%%%%%%%%%%%%%%%%%%%%%%%%%%%%%%%%%%%%%%%%%%%%%%%%%%%%%%%%%%%%%%%%%%

\begin{lemma}(cf. \cite[Remark 3.2.3]{ful-inter})\label{lem-comp-cc}
Let $\hh$ be an \oct and let $F$ be the respective \fgl
Let $E$ and $E^{\prime}$ be bundles over $X$
with Chern roots $\alpha_1, \ldots, \alpha_r$ and
$\alpha_1', \ldots, \alpha_s'$ respectively.
Let $E^\vee$ denote the dual of $E$.
Then
\begin{itemize}
\item[(i)]
  $\ch (E^\vee) = \prod_{i=1}^r \left(1 + \imath (\alpha_i) t \right)$.

\item[(ii)]
  $\ch (E \otimes E')
  =  \prod_{i,j} \left( 1 + (\alpha_i +_{F} \alpha_j') t\right)$.

\item[(iii)]
  $\ch \left( \wedge^l E \right)
  = \prod_{1 \leq i_1 < \cdots < i_l \leq r} \left( 1 + (\alpha_{i_1}
  +_{F} \cdots +_{F} \alpha_{i_l} ) t \right)$
  for $1 \leq l \leq r$.
\end{itemize}
\end{lemma}

\begin{proof}
(i) and (ii) follow from \eqref{equ-fgl&oct}.
To show (iii) assume that for $E = \oplus_{l=1}^r L_l$ the statement
holds.
Now consider the vector bundle $E \oplus L$
where $L$ is a line  bundle.
By \cite[$\S$4.13]{hir-tmag} we have a short exact sequence for all
$1 \leq l \leq r+1$
\[
0 \to \wedge^{l-1} (E) \otimes L  \to
  \wedge^l (E \oplus L) \to \wedge^l (E) \to 0.
\]
Set $\ch_1 (L) = \alpha_{r+1}$.
By (ii) and induction we obtain for $l \leq r$
\begin{multline*}
\quad \quad \quad
 \ch \left( \wedge^l (E \oplus L) \right)
=
  \ch \left( \wedge^{l-1} (E) \otimes L \right)
\cdot
  \ch \left( \wedge^l E \right)
\\
  \prod_{1 \leq i_1 < \cdots < i_{l-1} \leq r}
  \left( 1 + ( \alpha_{i_1} +_{F} \cdots +_{F}
  \alpha_{l-1}
  +_{F} \alpha_{r+1} ) t \right)
\\
  \prod_{1 \leq i_1 < \cdots < i_{l} \leq r}
  \left( 1 + ( \alpha_{i_1} +_{F} \cdots +_{F} \alpha_{l}
  ) t \right)
\\
=
  \prod_{1 \leq i_1 < \cdots < i_{l} \leq r+1}
  \left( 1 + ( \alpha_{i_1} +_{F} \cdots +_{F} \alpha_{l}
  ) t \right).
\end{multline*}
For $l = r+1$, since $\wedge^i E =0$ for all $i >rk(E)$,
by (ii) of the lemma and induction we conclude
\[
  \ch \left( \wedge^{r+1} (E \oplus L) \right)
=
  \ch \left( \wedge^r (E) \otimes L \right)
=
  1 + ( \alpha_1 +_{F} \cdots +_{F} \alpha_{r+1} ) \, t. \qedhere
\]
\end{proof}

Consider the following filtration on $\hh(X)$ by subgroups ($R$-linear) generated
by products of characteristic classes
\[
  \gamma^d\hh(X)
:=
  \langle\ch_1(L_1)\cdot\ldots\cdot \ch_1(L_l)
  \mid l\ge d,\; L_1,\ldots L_l
  \text{ are line bundles}\rangle.
\]

%%%%%%%%%%%%%%%%%%%%%%%%%%%%%%%%%%%%%%%%%%%%%%%%%%%%%%%%%%%%%%%%%%
\begin{lemma}\label{prop-pro}
Let $\hh$ be an \oct and let $F$ be the respective \fgl
Let $L_1, \ldots, L_r$ be line bundles over  $X$.
Then
\[
  \ch_r \left( \prod_{l=1}^r (1 - L_l^\vee ) \right)
=
  (-1)^{r-1} (r-1)! \cdot \prod_{l=1}^r \ch_1(L_l)
  \mod \gamma^{r+1}\hh(X).
\]
\end{lemma}

\begin{proof}
Denote $\alpha_l = \ch_1 (L_l)$, $1\le l \le r$ and
$E=\oplus_{l=1}^r L_l$.
We have
\begin{eqnarray}
  \ch \left( \prod_{l=1}^r (1 - L_l^\vee ) \right)
&=&
  \ch \left( \sum_{l=0}^r  (-1)^l \wedge^l E^\vee
  \right)
=
  \prod_{l=1}^r \ch \left( \wedge^l E^\vee
  \right)^{(-1)^l}
\notag
\\
&=&
  \prod_{l=1}^r \prod_{1 \leq i_1 < \cdots < i_l \leq r}
  \big( 1 + ( \imath(\alpha_{i_1}) +_{F} \cdots +_{F}
  \imath(\alpha_{i_l}) ) t \big)^{(-1)^l},
\end{eqnarray}
where the second equality holds by the property (c) of characteristic
classes and the last equality follows from
Lemma~\ref{lem-comp-cc}.(iii).

Since $\imath(\alpha_l)=-\alpha_l + O(2)$, the $r$th characteristic
class $\ch_r$ modulo $\gamma^{r+1}\hh(X)$ is given by the coefficient
at $t^r$ of the  following polynomial
\[
  \prod_{l=1}^r \prod_{1 \leq i_1 < \cdots < i_l \leq r}
  \big( 1 - (\alpha_{i_1} +\ldots + \alpha_{i_l} ) t \big)^{(-1)^l}
\]
and the desired formula then follows by \cite[Lem.~15.3]{ful-inter}.
\end{proof}

%%%%%%%%%%%%%%%%%%%%%%%%%%%%%%%%%%%%%%%%%%%%%%%%%%%%%%%%%%%%%%%%%%
%%%%%%%%%%%%%%%%%%%%%%%%%%%%%%%%%%%%%%%%%%%%%%%%%%%%%%%%%%%%%%%%%%
\section{Characteristic classes and deformation maps}
%%%%%%%%%%%%%%%%%%%%%%%%%%%%%%%%%%%%%%%%%%%%%%%%%%%%%%%%%%%%%%%%%%
%%%%%%%%%%%%%%%%%%%%%%%%%%%%%%%%%%%%%%%%%%%%%%%%%%%%%%%%%%%%%%%%%%

In the present section we establish a connection between
$F$-exponents, characteristic classes and maps involving flag
varieties (see~\ref{geomthm}).

%%%%%%%%%%%%%%%%%%%%%%%%%%%%%%%%%%%%%%%%%%%%%%%%%%%%%%%%%%%%%%%%%%

\medskip

Let $G$ be a split simple simply connected linear algebraic group of
rank $n$ over a field $k$.
Fix a split maximal torus $T$ and a Borel subgroup $B$ so that
$T \subset B \subset G$.
Let $\BB = G/B$ be the variety of Borel subgroups of $G$ and
let $\La$ be the group of characters of $T$. Let $W$ be the Weyl group
of the root system associated to $G$. It acts by linear automorphisms
on $\Lambda$ via reflections. Let $t$ be the torsion index of $G$. Recall that $t=1$ if $G$ is of type $A_n$ or $C_n$, and $t$ is a power of 2 if $G$ is of type $B_n$, $D_n$ or $G_2$.

Let $F$ be a \fgl  over $R$. Consider the formal group ring $\fgr$ and
the characteristic map defined in \cite[\S6]{ca&za&pe-tor}:
\[
  \cc_F \colon \fgr \too \hgr.
\]
If $F$ corresponds to an \oct $\hh$ from
\cite[Thm.13.12]{ca&za&pe-tor}, then $\hgr\simeq \hh(\BB)$ and the map
$\cc_F$ is given by $x_\la \mapsto \ch_1(L(\la))$, where $L(\la)$ is
the associated line bundle.

%%%%%%%%%%%%%%%%%%%%%%%%%%%%%%%%%%%%%%%%%%%%%%%%%%%%%%%%%%%%%%%%%%
\begin{example}
Characteristic map for the Chow theory $\mathrm{CH}$,
i.e. corresponding to the additive \fgl, is given by
 \[
  \mathfrak{c}_a \colon \Z [[\La]]_{F_a} \too \mathrm{CH}(\BB),
\quad
  x_\la \mapsto \mathrm{c}_1 \left( L(\la) \right).
\]
Hence, restricting to $S^*(\La)$ via $\veps_a$
we recover the usual characteristic map  for Chow groups
\cite[\S1.5]{dem-des}.

Characteristic map for the Grothendieck $K_0$, i.e. corresponding to
the multiplicative \fgl, is given by
\[
\mathfrak{c}_m \colon\Z [[\La]]_{F_m} \too K_0(\BB),
\quad
x_\la \mapsto 1 -[L(\la)^\vee].
\]
Restricting to the integral group ring $\Z[\La]$ via $\veps_m$
we recover the usual characteristic map for $K_0$
\cite[\S1.6]{dem-des} which maps $e^\la$ to $[L(\la)]$.

Observe that the 
algebraic cobordism $\Omega$ defined over a field of characteristic zero
satisfies the properties of  \cite[Thm.13.12]{ca&za&pe-tor},
therefore, we have the characteristic map $\cc_U\colon \LL[[\La]]_u
\to \Omega(\BB)$ defined by $x_\la \mapsto \mathrm{c}^\Omega_1(L(\la))$.
\end{example}

%%%%%%%%%%%%%%%%%%%%%%%%%%%%%%%%%%%%%%%%%%%%%%%%%%%%%%%%%%%%%%%%%%

\begin{theorem}\label{geomthm}
Let $\hh$ be an \oct as considered in
\cite[Thm.~13.12]{ca&za&pe-tor}. Assume that the coefficient ring $R$
of $\hh$ has characteristic 0. Let $F$ be the respective \fgl
over $R$.
Then there is a commutative diagram of $R$-modules
\[
  \xymatrix{
            \IM^{(d)}  \ar[r]^-{\cc_m^{(d)}}
             \ar[d]_{ (-1)^{d-1} (d-1)! \cdot \Phi^{F_m\to F}_{d}} &
            \gamma^{(d)} K_0(\BB) \ar[d]^{\ch_d}   \\
            \IF^{(d)} \ar[r]^-{\cc_F^{(d)} } &
            \gamma^{(d)}\hh(\BB)
           }
\]
where $\Phi^{F_m\to F}_{d}$ is the deformation map, $\ch_d$ is
the $d$-th characteristic class and
$\gamma^{(d)}\hh(X)=\gamma^d\hh(X)/\gamma^{d+1}\hh(X)$ is the
$d$-th subsequent quotient.

If $\frac{1}{t}\in R$, then we have $\ker \cc_F^{(d)} =  (\IF^W)^{(d)}$.
\end{theorem}

\begin{proof}
The commutativity of the diagram follows from Lemma~\ref{prop-pro}.
Now, suppose that $\frac{1}{t}\in R$, by \cite[Thm.~6.9]{ca&za&pe-tor} there is an exact
sequence of $R$-modules
\[
\xymatrix{
            0 \ar[r] &
          (\IF^W)  \ar[r] &
          (\fgr) \ar[r]^{\cc_F} &
            \hh(\BB) \ar[r] &
            0
 }.
\]
Passing to the subsequent quotients of the $\IF$-adic filtration we
obtain the exact sequence
\[
\xymatrix{
            0 \ar[r] &
          (\IF^W)^{(d)} \ar[r] &
          \IF^{(d)} \ar[r]^-{\cc_F^{(d)}} &
            \gamma^{(d)}\hh(\BB) \ar[r] &
            0
 }.\qedhere
\]
\end{proof}

\begin{corollary} For every $d\ge 0$ there is a group isomorphism
\[
  \mathrm{CH}^d(\BB) \otimes R_\qq
\simeq
  \gamma^{(d)}\hh(\BB)\otimes \qq.
\]
\end{corollary}

\begin{proof}
Follows from the isomorphism $\log_F^\star\colon (\IA^W)^{(d)}_{R_\qq}
\simeq (\IF^W)^{(d)}_\qq$ on the kernels, where $F$ is the respective
\fgl
\end{proof}

%%%%%%%%%%%%%%%%%%%%%%%%%%%%%%%%%%%%%%%%%%%%%%%%%%%%%%%%%%%%%%%%%%
%%%%%%%%%%%%%%%%%%%%%%%%%%%%%%%%%%%%%%%%%%%%%%%%%%%%%%%%%%%%%%%%%%
\section{Exponents of root systems. Applications.}\label{sec:t2}
%%%%%%%%%%%%%%%%%%%%%%%%%%%%%%%%%%%%%%%%%%%%%%%%%%%%%%%%%%%%%%%%%%
%%%%%%%%%%%%%%%%%%%%%%%%%%%%%%%%%%%%%%%%%%%%%%%%%%%%%%%%%%%%%%%%%%

In this section we describe some possible generators of $\IF^W$
and the exponents $\tau_d^{F\to F'}$ of the action of the Weyl group $W$
on the weight lattice $\La$ for root systems of types
$A_n$ ($n\ge 1$), $C_n$ ($n\ge 2$), $B_n$ ($n\ge 3$), $D_n$ ($n\ge 4$) and $G_2$.
In the present section $R$ is always a ring of characteristic different from $2$.
Recall that the torsion index $t$ of $\mathfrak{D}$ is  1 for types $A_n$ and $C_n$ and a power of 2 for types $B_n$, $D_n$ and $G_2$.
%%%%%%%%%%%%%%%%%%%%%%%%%%%%%%%%%%%%%%%

\begin{definition}
Given a root system of classical Dynkin type $\mathfrak{D}$ of rank $n$
we define the $W$-invariant elements
$\Th_d^\mathfrak{D} \in  \fgr^W\cap \IF$ as follows.

\smallskip
For the type $A_n$ we set
$\Th^A_d$ to be the $(d+1)$-th elementary symmetric polynomial in $\{x_{e_j}\}$, where $\{e_j\}_{j=1}^{n+1}$
is the standard basis of $\bbR^{n+1}$ on which the
Weyl group $W$ acts by permutations \cite[p.5]{hum}. 

\smallskip
For the type $B_n$ or $C_n$ we set
$\Th^B_d$ to be the $d$-th elementary symmetric polynomial in $\{x_{e_j}x_{-e_j}\}$, 
where $\{e_j\}_{j=1}^n$ is the standard basis of $\bbR^n$ on which the Weyl
group acts by permutations and sign changes \cite[p.5]{hum}.

\smallskip
For the type $D_n$ we set
$\Th_d^D:=\Th_d^B$ for $d=1\ldots n-1$ and
$\Th^D_n:=\prod_{i=1}^n(x_{e_i}-x_{-e_i})$, where $\{e_j\}_{j=1}^n$ is the standard basis of $\bbR^{n}$ on which  the Weyl group acts
by permutations and sign changes which
involve an even number of signs \cite[p.5]{hum}.

\smallskip
We will simply write  $\Th_d$, if the Dynkin type is clear for the
context.
\end{definition}

%%%%%%%%%%%%%%%%%%%%%%%%%%%%%%%%%%%%%%%

%%%%%%%%%%%%%%%%%%%%%%%%%%%%%%%%%%%%%

If $f\in \IF^d\setminus \IF^{d+1}$, we say that $\deg f=d$.
For instance, reducing to the additive case via the isomorphism
\eqref{eq-iso-slices} we obtain that
$\deg \Theta_d^A=d+1$ for $d=1\ldots n$,
$\deg \Theta_d^B=\deg\Th_d^C=2d$ for $d=1\ldots n$,
$\deg \Theta_d^D=2d$ for $d=1\ldots n-1$ and
$\deg\Theta_n^D=n$. Moreover, if $F=F_a$, the g.c.d. of the coefficients of $\Th^B_i$ and $\Th_i^D$ in $S^*(\La)$ is 1 if $2\le i\le n-1$, and is $2$ if $i=1$, and that of $\Th_n^B$ (resp. $\Th_n^D$) is 1 (resp. $2^n$).

\smallskip

 Let $(\fgr^W)^{(d)}=(\fgr^W\cap \IF^d)/(\fgr^W\cap
\IF^{d+1})$ denote the $d$-th homogeneous component of $\fgr^W$.
Given a $n$-tuple of non-negative integers $\al=(\al_1,...,\al_n)$ we
denote $\Th(\al):=\prod_{i=1}^n\Th_i^{\al_i}$,
$|\al|:=\sum_{i=1}^n\al_i \deg \Th_i $, where $\deg \Th_i$ is
the degree of $\Th_i$. Observe that $\deg \Th(\al)=|\al|$.
Let $\langle\Th(\al)\rangle_{|\al|=d}$ denote the $R$-linear span of
all $\Th(\al)$ such that $|\al|=d$.

\begin{lemma} \label{lem:Fainv} For classical Dynkin types, we have 
$$S^*(\Lambda)^W\otimes \Z[1/t]=\Z[1/t][\Th_1,...,\Th_n].$$
The set $\{\Th_1,...,\Th_n\}$ is called a set of basic invariants.
\end{lemma}
\begin{proof}By \cite[\S6]{dem-inv}, we know that the left hand side
  is a polynomial ring in $n$ generators $\{h_1,...,h_n\}$. 
The degrees of $h_i$ are uniquely determined by the Dynkin type.

If $\mathfrak{D}$ is of type $A_n$, then by \cite[Ch.I (2.4)]{Mac}, the conclusion follows.

If $\mathfrak{D}$ is of type $C_n$, since the Weyl group acts on $\{{e_i}\}$ by sign changes and permutations, so each $W$-invariant element belongs to $\Z[e_1^2,...,e_n^2]$ and is symmetric under permutations, so similar as the type $A_n$ case, we get the conclusion. 

The proof for the type $B_n$ is similar to that of the $C_n$ after inverting 2, noticing that $\Z[1/2][e_1,...,e_n]=\Z[1/2][\om_1,...,\om_n]$.

Now consider the type $D_n$. Let $q_{2i}:=\sum x_{e_j}^{2i}$ be the power sum of degree $2i$, and denote $\mathcal{Q}=\{q_2,...,q_{2n-2}, p_n\}$, $\mathcal{T}=\{\Th_1,...,\Th_{n-1},p_n\}$, $\mathcal{H}=\{h_1,...,h_n\}$ and $\mathcal{E}=\{e_1,...,e_n\}$. Denote by $|\partial A/\partial B|$ the Jacobian of $A$ with respect to $B$. By \cite[3.12]{hum}, $\mathcal{Q}$ is a set of basic invariants over $\mathbb{Q}$, and $$|\partial\mathcal{Q}/\partial\mathcal{E}|=(-2)^{n-1}(n-1)!\prod_{i<j}(e_i^2-e_j^2).$$ From the Newton's identities we see that $|\partial \mathcal{Q}/\partial \mathcal{T}|=(-1)^{n-1}(n-1)!$, so  $|\partial\mathcal{T}/\partial\mathcal{E}|=2^n\prod_{i<j}(e_i^2-e_j^2)$, and $\mathcal{T}$ is a set of basic invariants over $\mathbb{Q}$. Hence, $|\partial \mathcal{T}/\partial \mathcal{H}|$ is a rational number. Moreover, it divides $|\partial\mathcal{T}/\partial\mathcal{E}|$ in $\Z[1/2][h_1,...,h_n]$. So $|\partial \mathcal{T}/\partial \mathcal{H}|$ is a power of 2. Therefore, $\mathcal{T}$ is a set of basic invariants over $\Z[1/2]$. 
\end{proof}

\begin{lemma}\label{lem:invind} Let $\mathfrak{D}$ be of type $B_n$ (resp. type $D_n$), and $d\ge 1$ (resp. $n>d\ge 1$). Let $F=F_a$, and view $\Th_i$ as in $\Z[e_1,...,e_n]\subsetneq \Z[[\La]]_{F_a}$, then  the set $\{\Th(\al)\}_{|\al|=d}$ is linearly independent modulo any positive integer $M$.
\end{lemma}

\begin{proof} Firstly, let $s_1,...,s_n$ be elementary symmetric polynomials in $\Z[e_1,...,e_n]$. Then its Jacobian is $\prod_{i<j}(e_i-e_j)$, so $s_1,...,s_n$ are algebraically independent over $\Z/M\Z$. In particular, it shows that the set $\{f=s_1^i\cdots s_n^{i_n}|\deg f=d\}$ is linearly independent over $\Z/M\Z$. 

If $\mathfrak{D}$ is of type $B_n$, we identify the set $\{\Th_i\}_{i=1}^n\subset \Z[e_1,...,e_n]$ as the set of elementary symmetric polynomials in $\{e_j^2\}_{j=1}^n$, so  the conclusion follows.

For the type $D_n$, since $d<n$, the elements involved are
$\Th_1,...,\Th_{n-1}$, so the conclusion follows by the same arguments. 
\end{proof}
\begin{remark} In Lemma \ref{lem:invind} the set $\{\Th(\al)\}_{|\al|=d}$ may not be linearly independent in $\Z[\om_1,...,\om_n]$ modulo  certain $M$. For example, $\Th_1/2\in \Z[\om_1,...,\om_n]$ but $\Th_1/2\not\in \Z[e_1,...,e_n]$, so $\Th_1$ is linearly dependent in  $\Z/2[\om_1,...,\om_n]$ and linearly independent in $\Z/2[e_1,...,e_n]$.
\end{remark}

\begin{lemma} \label{lem:prinv}Suppose $\mathfrak{D}$ is of type $B_n$ (resp. $D_n$), and let $d\ge 2$ (resp. $d>n \ge 2$). Then 
\begin{itemize}
\item[(1)] $2^d\cdot (R[[\La]]_F^W)^{(d)}\subset \langle \Th(\al)\rangle_{|\al|=d}$,
\item[(2)] $
2^d\cdot (\IF^W)^{(d)}\subseteq
\big\{\sum_{\deg \Th_i\le d} g_i\Th_i\mid g_i\in \IF^{(d-\deg \Th_i)}\big\}$.
\end{itemize}
If $1/2\in R$, then the ``$\subset$'' in (1) and (2) can be replaced by ``$=$'' after removing $2^d$.
\end{lemma}
\begin{proof} The proof for the type $D_n$ with $d<n$ is similar to that for the type $B_n$, so we only consider the latter case.

(1) Suppose $\mathfrak{D}$ is of type $B_n$. We first prove it for $F=F_a$. It is equivalent to show that $$2^d\cdot S^d_R(\La)^W\subset \langle \Th(\alpha)\rangle_{|\al|=d}.$$
For any $f\in S^d_R(\La)^W$, expressing each $\om_i$ in terms of $e_j$, we see that $f=\tilde f/2^d$ for some  $\tilde f\in R[e_1,...,e_n]^W$ of degree $d$. By Lemma \ref{lem:Fainv}, 
$$\tilde f=\sum_{|\al|=d} \frac{a_\alpha}{b_\alpha}\Th(\al) \text{ for some }a_\al, b_\al, $$
where $b_\al$ is some power of $2$.
Clearing the denominators, we see that $2^r\cdot \tilde f=\sum a_\alpha'\Th(\alpha)$ for some $r,a_\alpha'\in \Z$. Since $2^r$ divides the left hand side, by Lemma \ref{lem:invind}, $2^r$ divides $a_\alpha'$ for all $\al$. Therefore, $r=0$ and we obtain $2^d\cdot f=\tilde f=\sum a_\alpha' \Th(\alpha)$.

For a general \fgl $F$ we have (see \eqref{eq-iso-slices})
\[
g\in (\fgr^W)^{(d)}\subset
(\IF^{(d)})^{W}\underset{\sim}{\overset{\Phi_d^{F\to F_a}}\too}
(\IA^{(d)})^W=S^d_R(\La)^W.
\]
Hence
\[
2^d\cdot \Phi_d^{F\to F_a}(g)=\sum_{|\al|=d}
c_\alpha\Th(\al)\in \IA^{(d)}\quad\text{for some }
c_{\al}\in R.
\]
Applying the inverse $\Phi^{F_a\to F}$ to  both sides we obtain
\[
2^d\cdot g=\sum_{|\al|=d} c_{\al}\Th(\al).
\]

(2)
 The group $(\IF^W)^{(d)}$ consists of $R$-linear combinations of
products $h=gf$ with $f\in \fgr^W\cap \IF$ of degree $k$ ($2\le k\le
d$) and $g\in \fgr$ of degree $g=d-k$. So by (1), 
\[
2^{d}h=2^{d}gf=g\sum_{|\al|=k}b_{\al}\Th(\al)=\sum_{\deg \Th_i\le d}g_i\Th_i.
\]

If $1/2\in R$, the conclusion follows immediately. 
\end{proof}
\begin{remark}For the type $D_n$, if $1/2\in R$, then the conclusion of the last part in Lemma \ref{lem:prinv} holds for $d\ge n$ as well. 
\end{remark}
\medskip

We now consider the $G_2$ type. Let $\{e_1,e_2,e_3\}$ be the standard basis of $\bbR^3$.
The fundamental weights are
$\om_1=e_3-e_1$ and $\om_2=2e_3-e_1-e_2$,
and the Weyl group associated to $G_2$ is the dihedral group $D_6$.
Define
\[
  \Th_1^{G_2} =
  x_{\om_1}x_{-\om_1}
+
  x_{\om_2-\om_1}x_{\om_1-\om_2}
+
  x_{\om_2-2\om_1}x_{2\om_1-\om_2}\text{ and }
\]
\[
  \Th_2^{G_2} =
  x_{\om_1}x_{-\om_1}
  x_{\om_2-\om_1}x_{\om_1-\om_2}
  x_{\om_2-2\om_1}x_{2\om_1-\om_2}.
\]
A direct computation shows that $\Th^{G_2}_1$, $\Th_2^{G_2}$ are both
$W$-invariant, and $\deg \Th_1^{G_2}=2, \deg \Th_2^{G_2}=6$. 

\begin{lemma}\label{lem:G2}
For the $G_2$ type, we have 
$S^*_{\Z[1/2]}(\La)^W=\Z[1/2][\Th_1^{G_2},\Th_2^{G_2}].$
\end{lemma}
\begin{proof}
By \cite[\S6]{dem-inv}, $S_{\Z[1/2]}^*(\La)^W$ is a polynomial ring generated by two homogeneous polynomials of degree 2 and 6, respectively.  Suppose $\{h_1,h_2\}$ is a pair of generators. Note that
the Jacobian  $$\Big|\frac{\partial \Th^{G_2}_i}{\partial {\om_j}}\Big|=-4\om_1\om_2(\om_1-\om_2)(2\om_1-\om_2)(3\om_1-2\om_2)(3\om_1-\om_2).$$
So $\Th_1^{G_2}$ and $\Th_2^{G_2}$ are
algebraically independent. So by \cite[\S3.11]{hum}, $\{\Th^{G_2}_1,\Th^{G_2}_2\}$ is a set of generators of the invariants over $\mathbb{Q}$. Therefore, the Jacobian $|\frac{\partial \Th^{G_2}_i}{\partial h_j}|$ is a rational number  dividing $|\frac{\partial \Th^{G_2}_i}{\partial {\om_j}}|$, so it is a unit in $\Z[1/2]$. The proof is finished. 
\end{proof}

\begin{lemma}\label{lem:G2inv} For each even integer $d\ge 1$, let $r_d$ to be the remainder of $d$ modulo 6, and let  $\zeta_d=\frac{3[d/6]([d/6]+1)}{2}+\frac{r_d}{2}$. If $\mathfrak{D}$ is of type $G_2$, then 
\begin{itemize}
\item[(1)] $2^{\zeta_d}\cdot (R[[\La]]_F^W)^{(d)}\subset \langle \Th(\al)\rangle_{|\al|=d}\subset R[[\La]]_F.$
\item[(2)] $
2^{\zeta_d}\cdot (\IF^W)^{(d)}\subseteq
\{\sum_{\deg \Th_i\le d} g_i\Th_i\mid g_i\in \IF^{(d-\deg \Th_i)}\}.
$
\end{itemize}
If $1/2\in R$, then the ``$\subset$'' can be replaced by ``$=$'' after removing $2^{\zeta_d}$. 
\end{lemma}
\begin{proof} 
(1) Suppose $F=F_a$. For any $f\in S^d_R(\La)^W$, then $f\in R[1/2][\Th_1,\Th_2]$, so by Lemma \ref{lem:G2}, we can write 
$2^r\cdot f=\sum_{i=0}^{[d/6]}a_i\Th_2^i\Th_1^{d/2-3i}$ with $a_i\in R$. To prove the lemma, it suffices to show that the smallest $r$ in the equation is less than or equal to $\zeta_d$, i.e., to show that if $r>\zeta_d$, then $2|a_i$ for all $i$. Notice that $\frac{\Th_1}{2}\in S_R^*(\La)$, but $\frac{\Th_1}{4}\not\in S_R^*(\La)$ and $\frac{\Th_2}{2}\not\in S_R^*(\La)$.

We prove it  for  $d=6k$ first, in which case $\zeta_d=\frac{3k(k+1)}{2}$. We proceed by induction on $k$.  If $k=1$, then $r>\zeta_6=3$ and 
$2^rf=a\Th_2+b\Th_1^3.$
 Letting $\om_1=\om_2=1$, we have $\Th_2(1,1)=0$ and $\Th_1(1,1)=2$, so $2^r|b\cdot 8$, hence, $2^{r-3}|b$. So $2^{r-3}|a$ since $\Th_2/2\not\in S_R^*(\La)$.
 
 Assume the conclusion for $k_0-1$, then for $k_0$, note that $\zeta_{6k_0}=\zeta_{6k_0-6}+3k_0$. The equation becomes 
 $$2^rf=\sum_{i=0}^{k_0}a_i\Th_2^i\Th_1^{3k_0-3i},$$
 and $r>\zeta_{6k_0}$. Letting $\om_1=\om_2=1$, we see that $2^r|a_{0}2^{3k_0}$, so $2^{r-3k_0}|a_0$. Therefore, 
 $$2^{r-3k_0}|\sum_{i=1}^{k_0}a_i\Th_2^i\Th_1^{3k_0-3i}=\Th_2\sum_{j=0}^{k_0-1}a_{j+1}\Th_2^{j}\Th_1^{3k_0-3-3j}.$$
 So $2^{\zeta_{6k_0-6}}|2^{r-3k_0}|\sum_{j=0}^{k_0-1}a_{j+1}\Th_2^j\Th_1^{3k_0-3-3j}$. By induction,  it implies that $2|a_j$ for $j=1,...,k_0$.
 
 We then assume that $d=6k+r_d$ with $r_d>0$ the remainder of $d$ modulo 6. Then $r>\zeta_d=\zeta_{6k}+r_d/2$, and we have 
 $$2^rf=\Th_1^{r_d/2}\sum_{i=0}^ka_i\Th_2^i\Th_1^{3k-3i}.$$
 So $2^{r-\tfrac{r_d}{2}}|\sum_{i=0}^ka_i\Th_2^i\Th_1^{3k-3i}$. The conclusion then follows from the first part.

The proof of the rest is similar to that of \ref{lem:prinv}.
\end{proof}

\begin{example} \label{ex:AC}If $\mathfrak{D}$ is of type $A_n$ or $C_n$, then the torsion index is 1, and $\{e_i\}_{i=1}^n$ is a basis of $\La$. Following from the same idea of Lemma \ref{lem:prinv},  we have 
$$(R[[\La]]_F^W)^{(d)}= \langle \Th(\al)\rangle_{|\al|=d}$$
and 
$$(\IF^W)^{(d)}=
\big\{\sum_{\deg \Th_i\le d} g_i\Th_i\mid g_i\in \IF^{(d-\deg \Th_i)}\big\}.$$

\end{example}

We are ready now to prove the main result of this section

\begin{theorem}\label{ABD} 
 Let $F$ and $F'$ be formal group laws over a ring $R$ of
  characteristic different from 2.
 For the type $B_n$ with $n\ge 3$ and $d\ge 1$ (resp. $D_n$ with $n\ge 4$ and $n>d\ge 1$), we have
  $\tau_d^{F\to F'}\mid 2^d$. For the type $G_2$, $\tau_d^{F\to F'}|2^{\zeta_d}$, where $\zeta_d$ was defined in Lemma \ref{lem:G2inv}. In particular, for type $A_n$, $B_n$, $C_n$, $D_n$  or $G_2$, if $\frac{1}{t}\in R$,  then $\tau_d^{F\to F'}=1$ for every $d\ge 1$.

\end{theorem}

\begin{proof}Suppose that $\mathfrak{D}$ is of type $B_n$ or $D_n$. 
 By Lemma~\ref{lem:prinv}.(2) we have
\[
2^d\cdot (\IFp^W)^{(d)}\subseteq \{\sum_{\deg \Th_i\le d}
g_i\Th_{i,F'}\mid g_i\in \IFp^{(d-\deg \Th_i)}\}\subseteq
\Phi_d^{F\to F'}((\IF^W)^{(d)}),
\]
as $\Phi_{d}^{F\to F'}(g_i\Th_{i,F})=g_i\Th_{i,F'}$, where $\Theta_{i,F}$ and $\Th_{i,F'}$ are
the respective $W$-invariant elements for the \fgl $F$ and $F'$.
So $\tau_d^{F\to F'}\mid 2^d$ for any $F$ and $F'$. 

The part for $G_2$ follows similarly from Lemma \ref{lem:G2inv}.

The third statement follows immediately, after using Example \ref{ex:AC} if $\mathfrak{D}$ is of type $A_n$ or $C_n$. 
\end{proof}

Let  $\hh$ be an
  oriented cohomology theory satisfying the condition in \cite[Theorem
  13.12]{ca&za&pe-tor} with the corresponding formal group law $F$ over
the coefficient ring $R$ of characteristic zero.
Let
  $\mathfrak{B}$ be the variety of Borel subgroups 
 and let $\mathfrak{c}_F^{(d)}\colon \IF^{(d)}\to \gamma^{(d)}\hh(\mathfrak{B})$
denote the characteristic map.
Identifying the $R$-module $\IF^{(d)}$ with
$\calI_{F'}^{(d)}$ via $\Phi_d^{F\to F'}$, setting $A=\IF^{(d)}=\calI_{F'}^{(d)}$,
$M=\ker c_F^{(d)}$, $M'=\ker c_{F'}^{(d)}$, $N=(\IF^W)^{(d)}$,
$N'=(\calI^W_{F'})^{(d)}$, and assuming that there are exist exponents
$\tau$ (resp. $\tau'$) of
$N$ in $M$ (resp. $N'$ in $M'$) we obtain the following
relations between the annihilators of torsion elements in
$\gamma^{(d)}\hh(\mathfrak{B})=A/M$ and $\gamma^{(d)}\hh'(\mathfrak{B})=A/M'$:

\begin{lemma} Let $z\in A$ be such
  that its image $\bar z$ in $A/M$ is torsion or zero. Then its image
  $\bar z'$ in $A/M'$ is torsion or zero. Let $a$ and $a'$
  denote the smallest positive annihilators of $\bar z$ and $\bar z'$
  respectively.

Then
$a' \mid \tau^{F'\to F}_d \tau a$ and $a \mid \tau^{F\to F'}_d \tau' a'$.
\end{lemma}

\begin{proof}
(a) Since $a \bar z=0$, $az \in M$, therefore, $\tau az\in N$ and
$\tau^{F'\to F}_d \tau a z \in N'\subset M'$. So $\tau^{F'\to F}_d \tau a
\bar z'=0$ which implies that $a' \mid \tau^{F'\to F}_d \tau a$.
\end{proof}

\begin{corollary} \label{torsapp} Assume that
  $\gamma^{(d)}\hh'(\mathfrak{B})$ has no torsion and that the
  exponent $\tau'$ exists. Then the torsion of
  $\gamma^{(d)}\hh(\mathfrak{B})$ is annihilated by $\tau_d^{F\to F'} \tau'$.
\end{corollary}

\begin{corollary}\label{torsch} Let $G$ be a split, simple simply connected linear
  algebraic group of type $A_n$, $B_n$, $C_n$ , $D_n$ or $G_2$. 

If $\frac{1}{t}\in R$, then  $\gamma^{(d)}\hh(\mathfrak{B})$ is
torsion free for every $d\ge 1$ and the deformation map $\Phi^{F\to
  F_a}_d$ induces an isomorphism $\gamma^{(d)}\hh(\mathfrak{B}) \simeq \gamma^{(d)}\CH(\mathfrak{B};R)$.
\end{corollary}

\begin{proof}
To see that $\gamma^{(d)}\hh(\mathfrak{B})$ is torsion free, take
$\hh'(-)=\CH(-;R)$ and apply Theorem~\ref{ABD} and Corollary~\ref{torsapp}.
Observe that $\ker \mathfrak{c}_a^{(d)}$ is finitely generated over $R$ so that
the exponent $\tau'$ exists.
  
Let $z\in \ker \mathfrak{c}_F^{(d)}$. By Lemma \ref{lem:prinv},  $z=\sum
g_i\Th_i$ for some $g_i\in \IF^{(d-\deg \Th_i)}$. Since $\Phi_d^{F\to
  F_a}$ maps $\prod_{i=1}^dx_{\la_i}\in \IF^{(d)}$ to
$\prod_{i=1}^dx_{\la_i}\in \calI_{a}^{(d)}$, we obtain
$$\Phi_d^{F\to F_a}( z)=\sum g_i\Th_i\in (\calI_{a}^W)^{(d)}=\ker \mathfrak{c}_{a}^{(d)}.$$
Therefore, $\Phi_d^{F\to F_a}$ induces an isomorphism $\ker \mathfrak{c}_F^{(d)}\cong \ker \mathfrak{c}_{a}^{(d)}$, and hence an isomorphism on the cokernels  $\gamma^{(d)}\hh(X)\cong \gamma^{(d)}\CH(X;R)$.
\end{proof}

\paragraph{\bf Acknowledgments}
The first author was supported by NSERC grants of D.~Daigle and the
second author.
The second author was supported by the NSERC Discovery grant
385795-2010, NSERC DAS grant 396100-2010 and the Early Researcher Award (Ontario).
The third author was supported by the Fields Institute,
I.H.E.S., SFB/Transregio 45,
University of Mainz and NSERC grant of the second author.

%%%%%%%%%%%%%%%%%

%%%%%%%%%%%%%%%%%%%%%%%%%%%%%%%%%%%%%%%%%%%%%%%%%%%%%%%%%%%%%%%%%%
%%%%%%%%%%%%%%%%%%%%%%%%%%%%%%%%%%%%%%%%%%%%%%%%%%%%%%%%%%%%%%%%%%

%%%%%%%%%%%%%%%%%%%%%%%%%%%%%%%%%%%%%%%%%%%%%%%%%%%%%%%%%%%%%%%%%%
%%%%%%%%%%%%%%%%%%%%%%%%%%%%%%%%%%%%%%%%%%%%%%%%%%%%%%%%%%%%%%%%%%

\bibliographystyle{plain}

\begin{thebibliography}{10}

\bibitem{ba&ne&za-ch}
Sanghoon Baek, Erhard Neher, and Kirill Zainoulline.
\newblock {\em Basic polynomial invariants, fundamental representations and the
  {C}hern class map.}
\newblock { Documenta Math.}, 17:135--150, 2012.

\bibitem{ba&za&zh}
Sanghoon Baek, Kirill Zainoulline, and Changlong Zhong.
\newblock {\em On the torsion of {C}how groups of twisted {S}pin-flags.}
\newblock {Math. Res. Letters, to appear.}

\bibitem{ca&za&pe-tor}
Baptiste Calm\`es, Victor Petrov, and Kirill Zainoulline.
\newblock {\em Invariants, torsion indices and oriented cohomology of complete
  flags.}
\newblock {Ann. Sci. Ec. Norm. Sup. (4)} 46(3): 36pp, 2013.

\bibitem{dem-inv}
Michel Demazure. 
\newblock {\em Invariants sym\'etriques entiers des groupes de Weyl et torsion.} Invent. Math. 21 (1973), 287-301.
\bibitem{dem-des}
Michel Demazure.
\newblock {\em D\'esingularisation des vari\'et\'es de {S}chubert
  g\'en\'eralis\'ees.}
\newblock { Ann. Sci. Ec. Norm. Sup. (4)}, 7(4):53--88, 1974.

\bibitem{ful-inter}
William Fulton.
\newblock {\em Intersection theory. 2nd ed.}, volume~2 of { Ergebnisse der
  Mathematik und ihrer Grenzgebiete. 3. Folge}.
\newblock Springer-Verlag, Berlin, 1998.

\bibitem{ga&ki-gamma}
Skip Garibaldi and Kirill Zainoulline.
\newblock {\em The gamma-filtration and the {R}ost invariant.}
\newblock {J. Reine Angew. Math., to appear}, 2013.

\bibitem{hir-tmag}
Friedrich Hirzebruch.
\newblock {\em Topological methods in algebraic geometry}.
\newblock Classics in Mathematics. Springer-Verlag, Berlin, reprint of the 2nd,
  corr. print. of the 3rd ed. edition, 1995.

\bibitem{hum}
James E. Humphreys.
\newblock {\em Reflection groups and coxeter groups.}
\newblock In { Cambridge Studies in Advanced Math.}, number~29. Cambridge
  Univ. Press, 1990.

\bibitem{lev&mor-book}
Marc Levine and Fabien Morel.
\newblock {\em Algebraic cobordism}.
\newblock Springer Monographs in Mathematics. Springer-Verlag, Berlin, 2007.


\bibitem{Mac} Ian G. Macdonald.
\newblock{\em Symmetric functions and Hall polynomials.}
\newblock{ Oxford Univ. Press, Oxford, 1979.}


\bibitem{pa&sm}
Ivan Panin.
\newblock {\em Oriented cohomology theories of algebraic varieties.}
\newblock { K-Theory,  30(3):265--314, 2003.}

\bibitem{Ell}
Joseph Silverman.
\newblock {\em The arithmetic of elliptic curves. 2nd ed.}
\newblock Graduate Texts in Math. 106. Springer, 2009.

\end{thebibliography}

%%%%%%%%%%%%%%%%%%%%%%%%%%%%%%%%%%%%%%%%%%%%%%%%%%%%%%%%%%%%%%%%%%
%%%%%%%%%%%%%%%%%%%%%%%%%%%%%%%%%%%%%%%%%%%%%%%%%%%%%%%%%%%%%%%%%%

\end{document}